%% file: main.tex
\documentclass[11pt, letterpaper]{amsart}
\usepackage[left=1in,right=1in,bottom=0.5in,top=0.8in]{geometry}
\usepackage{amsfonts}
\usepackage{amsmath, amssymb}
\usepackage{graphicx}
\usepackage[font=small,labelfont=bf]{caption}
\usepackage{epstopdf}
\usepackage[pdfpagelabels,hyperindex]{hyperref}
\usepackage{xcolor}
\usepackage{amsthm}
\usepackage{float}
\usepackage{pgfplots}
\usepackage{listings}
\usepackage{longtable}
\usepackage{mathrsfs}
\usepackage{tikz}                   % Drawing of everything
\usepackage{tikz-cd}
\usepackage{cleveref}
\usepackage{comment}
\usepackage{setspace}

\newtheorem{thm}{Theorem}[section]

\newtheorem{prop}[thm]{Proposition} 
\newtheorem{lem}[thm]{Lemma}

\theoremstyle{definition}
\newtheorem{defn}[thm]{Definition}

\newtheorem{rem}[thm]{Remark}

\theoremstyle{remark}

\newcommand{\V}{\mathfrak{V}} %linked nets
\newcommand{\Z}{\mathbb{Z}} %integer numbers
\newcommand{\R}{\mathbf{R}} %real numbers
\renewcommand{\k}{\mathbf{k}}
\newcommand{\Pf}{\mathbf{P}} %polytopes
\newcommand{\W}{\mathcal{W}}
\newcommand{\F}{\mathbf{F}} %face of a polytope

\renewcommand{\P}{\mathbb{P}} %projective space
\newcommand{\LP}{\mathbb{LP}} %Linked proj.space
\renewcommand{\a}{\mathfrak{a}} %arrow type 
\renewcommand{\t}{\textbf{t}} %type of a path

\newcommand{\ess}{\mathrm{Ess}} 
\newcommand{\e}{\mathfrak{e}}

\newcommand{\cd}{\mathrm{cd}}

\DeclareMathOperator{\im}{\text{Im}}
\definecolor{energy}{RGB}{114,0,172}
\definecolor{freq}{RGB}{45,177,93}
\definecolor{spin}{RGB}{251,0,29}
\definecolor{signal}{RGB}{203,23,206}
\definecolor{circle}{RGB}{217,86,16}
\definecolor{average}{RGB}{203,23,206}
\colorlet{shadecolor}{gray!20}
\pgfplotsset{compat=1.9}

\usepgflibrary{fpu}

\makeatletter
\let\c@equation\c@thm
\makeatother
\numberwithin{equation}{section}

\title{Polymatroidal Tilings and the Chow Class of Linked Projective Spaces}

\begin{document}

%\input{convexity/convexity}
%\begin{comment}

\author[]{EDUARDO ESTEVES}
\address{Instituto de Matemática Pura e Aplicada (IMPA), Rio de Janeiro, Brazil}
\email{esteves@impa.br}

\author[]{FELIPE DE LEON}
\address{Instituto de Matemática Pura e Aplicada (IMPA), Rio de Janeiro, Brazil}
\email{felipe.leon@impa.br}

\urladdr{}
%\subjclass[2024]{}
\title{Polymatroidal Tilings and the Chow Class of Linked Projective Spaces}
\thanks{The authors thank Omid Amini for helpful discussions. Felipe de Leon was supported by CAPES.}

\begin{abstract}

Linked projective spaces are quiver Grassmanians of constant dimension one of certain quiver representations, called linked nets, over special class of quivers, called $\Z^n$-quivers. They were recently introduced as a tool for describing schematic limits of families of divisors. They are subschemes of products of projective spaces of the same dimension. It is an open question whether they are degenerations of the (small) diagonal. 
We show that they have the Chow class of the diagonal. \end{abstract} 
\maketitle

\input{intro2}

\input{poly}

\input{Znquiversandlinkednets}
\input{linkednetsvectorspaces}

\input{linkedpro}

%\end{comment}

\bibliographystyle{plain}
\bibliography{ref.bib}

\end{document}

%% file: intro2.tex
A family of linear series on a family of smooth varieties degenerating to a singular variety $X$ degenerates to at least one linear series, if the family of varieties is semistable, but often infinitely many linear series, if $X$ is reducible. How to understand the collection of all limits of linear series? 

%Given a family of linear series on a family of smooth curves degenerating to a singular curve $X$, the family of of divisors associated to the liner series has as limit a collection of subschemes of $X$ of the same length. What is this collection?. If $X$ is irreducible, it is a subscheme of $\mathrm{Hilb_X}$ as follows from work by Almant and Kleiman \cite{ALTMAN198050}. What if $X$ is reducible?

If $X$ is a curve of compact type, there is the theory of limit linear series by Eisenbud and Harris \cite{Eisenbud1986}. More generally, if $X$ is a nodal curve, there is the theory by Osserman \cite{Osserman2019}, who organizes the limits of linear series as a quiver representation in the category of linear series on $X$. This approach was generalized and carried out in detail by Esteves, Santos and Vital \cite{esteves2024quiverrepresentationsarisingdegenerations,Esteves22072025} for the general case of a semistable family of varieties (and even more generally; see the end of the introduction). In particular, they showed that the limits of divisors along the family are parametrized by the image of a map $\LP(\V)\rightarrow\mathrm{Hilb}_X$, where $\V$ is the representation in the category of vector spaces underlying the limit representation in the linear series, and $\LP(\V)$ is its quiver Grassmannian of subrepresentations of constant dimension $1$.

%If $X$ is reducible, there are infinitely many linear series on $X$ that arise as limits along the family. These limits were studied by Eisenbud and Harris , as well as later by Osserman for each $X$ is a nodal curve of compact type, with two components in Osserman's case. In \cite{esteves2024quiverrepresentationsarisingdegenerations}, the authors observed that there is mention in the literature to the connection between a `limit linear series` and schematic limits of divisors until a work by Esteves and Osserman, \cite{Esteves2013}. However the limits were considered only for nodal curves with two components and a single node, and only as cycles. The general case was considered in \cite{esteves2024quiverrepresentationsarisingdegenerations}. 

Furthermore, they abstracted the properties of $\V$ arising from the degeneration to deal more generally with certain representations, called linked nets, of certain quivers, called $\Z^n$-quivers.
%In \cite{esteves2024quiverrepresentationsarisingdegenerations}, the authors show that a rank-$r$ liner series on the general fiber of a \textit{regular smoothing} of $X$ gives rise to two quivers representaions: a representation in the category of line bundles over $X$ and a subrepresentation $\V$ of the constant dimension $r+1$ of the induced representation $H^0(X,\mathfrak{L})$ in the category of vector spaces over $k$ obtained from $\mathfrak{L}$ by taking global sections. There are a natural scheme structure for the quiver grassmanaian $\LP(\V)$ of constant dimension and a natural morphisms $\LP(\V)\rightarrow \mathrm{Hilb}^d_X$ whose image is the collection of schematic limits of divisors associated to a degenerating family of linear series.

A quiver $Q$ is a $\Z^n$-quiver if it is endowed with a nontrivial partition $T=\lbrace A_0,\dots,A_n\rbrace$ of its arrow set in $n+1$ parts, called \textit{arrow types}, such that for each vertex there is a unique arrow of each type leaving it;  and each vertex is connected to each other by paths that do not contain arrows of all types, called \textit{admissible paths}, each such path having the same number of arrows of each type. (See  \Cref{section:2} for precise definitions and further details.) 
 
A representation $\V$ of $Q$ in the category of finite-dimensional vector spaces over a field $\k$ is said to be a \textit{linked net} if: \begin{enumerate}
    \item The composition $\varphi^\V_\gamma$ of the maps associated to each admissible path $\gamma$ depends only on the vertices the path connects, up to homothety;
    \item The composition $\varphi^\V_\gamma$ is zero if $\gamma$ is a non-admissible path;
    \item For any two admissible paths $\gamma_1$ and $\gamma_2$ leaving the same vertex and with no common arrow type, $\ker(\varphi^\V_{\gamma_1})\cap\ker(\varphi^\V_{\gamma_2})=0$
\end{enumerate} The representation $\V$ called \textit{pure} if the associated vector spaces $V_v$ have the same dimension. It is called \textit{exact} if the image of $\varphi_\gamma^\V$ is equal to the kernel of $\varphi^\V_\mu$ for any admissible paths $\gamma$ and $\mu$ whose concatenation $\mu\gamma$ is a circuit containing a single arrow of each type for all types.

$\Z^n$-quivers are infinite. We focus on the class of linked nets $\V$ that are \textit{finitely generated}, those for which there exists a finite subset of vertices $H$ of $Q$ such that for each vertex $v$, there is a vertex $u\in H$ such that $\varphi_\gamma$ is surjective for each admissible path connecting $u$ to $v$. In this case, $H$ can be chosen convex, that is, satisfying $P(H)=H$, where $P(H)$ is the \textit{hull} of $H$, i.e., the set of vertices that can be reached from $H$ by an admissible path avoiding each given arrow type.

 The \textit{linked projective space} $\LP(\V)$ is the quiver Grassmannian of subrepresentations of pure dimension $1$ of $\V$.  If $\V$ is pure of dimension $r+1$ and generated by a finite convex set $H$, then $\LP(\V)$ naturally embeds into a product of projective spaces $\prod_{v\in H} \P(V_v)$. 

The linked nets $\V$ arising from limit of linear series are exact and finitely generated. Furthermore, they are \textit{smoothable}. More precisely, $\V$ is the residue of a representation of $Q$ in the category of modules over a discrete valuation ring whose associated maps generically isomorphisms.

%More precisely, a \textit{general linked net} over a $\Z^n$-quiver $Q$ is a representation $\V$ of $Q$ such that\begin{enumerate}    \item $\varphi^\V_a$ is an isomorphism for each arrow $a$ of $Q$;    \item for each paths $\gamma_1$ and $\gamma_2$ connecting the same two vertices, $\varphi^\V_{\gamma_1}$ is a scalar multiple of $\varphi^\V_{\gamma_2}$. \end{enumerate} Let $R$ be a discrete valuation ring with residue field $\k$ and field of fractions $K$.Let $\mathfrak{M}$ be a representation of $Q$ in the category of free modules of a given rank $n$ over $R$. Assume the induced representation over $K$ is a general linked net and that over $\k$ is a linked net $\V$. We call $\mathfrak{M}$ a \textit{smoothing} of $\V$ over $R$ and say $\V$ is \textit{smoothable}. In this case, $\LP(\V)$ is a flat degeneration of the diagonal of a certain product of protective spaces of the same dimension. 

As mentioned in \cite{esteves2024quiverrepresentationsarisingdegenerations}, it is an
open question whether an exact finitely generated linked net $\V$ is smoothable. When $\V$ is smoothable, $\LP(\V)$ is the special fiber of a Mustafin variety and, therefore, is a flat degeneration of the (small) diagonal in $\prod_{v\in H}\P(V_v)$. In this paper, we establish a related but weaker property: that linked projective spaces have the Chow class of the diagonal. More precisely,
\vspace{0.1cm}

{\noindent\textbf{\Cref{thm:4.4}}\textit{ Let $\V$ be a nontrivial exact finitely generated linked net of vector spaces of dimension $r+1$ over a $\Z^n$-quiver $Q$, and let $H$ be its minimum set of generators. Then $H$ is convex and $$[\LP(\V)]=\sum_{q\in \Omega_{r}(\Z)} \prod_{v\in H} h_v^{r-q(v)}\in A^*\left(\prod_{v\in H}\P(V_v)\right),$$  where $\Omega_{r}(\Z)$ stands for the integer points in $$\Omega_{r}:=\left\lbrace q\in \R^H \,\big\vert \,q(v)\geq 0\,\, \forall v\in H \text{ and } \sum_{v\in H}q(v)=r\right\rbrace.$$}
    }
\vspace{0.1cm}

This is not sufficient to conclude that $\LP(\V)$ is a deformation of the small diagonal. Even having the same Hilbert polynomial as the diagonal does not suffice; see \cite{10.1093/imrn/rnp201}. 

We now proceed to outline the content in this paper. 

In \Cref{section:1}, we collect facts about polytopes and polyhedral tilings of $\Omega_{r+1}$.  The core of this section is \Cref{prop:1.3} which gives conditions under which certain families of polytopes form polyhedral tilings. The results presented are adapted versions of results due to E. Esteves and O.~ Amini in a forthcoming paper. For the sake of completeness, sketches of their proofs are included. 

In \Cref{section:2}, we introduce the notion of extreme vertices. They play an important role in showing that the minimum set of generators $H$ of $\V$ is convex. The hull condition, $P(H)=H$, has several incarnations in the literature. For example, for certain models of $Q$, it coincides with tropical convexity; see \cite[Prop. 5.8]{Esteves22072025}. More generally, it defines an abstract convexity on the set of vertices of $Q$; see \cite{poncet:hal-00922374}. In this sense, what we call extreme points coincide with the usual notion of extreme points coming from abstract convexity for the opposite quiver $Q^*$, instead of $Q$ itself. Our notion of extreme points is also related to the classical notion of $v$-reduced divisors on complete linear series on finite graphs; see \cite{BAKER2007766}, as well as to Osserman's notion of concentrated multidegrees; see \cite{Osserman+2019+57+88}, and both are equivalent to extremeness in our definition; see \Cref{rem:2.3}.

In \Cref{section:3}, we associate to $\V$ a family of polytopes, one polytope $\Pf_v$ for each $v\in H$, where $H$ is the minimum set of generators of $\V$. We show that, for each $v\in H$, the faces of $\Pf_v$ that intersect the relative interior of $\Omega_{r+1}$, are in bijection with the set of polygons in $H$ that contain $v$. 

\Cref{section:4} is devoted to showing \Cref{thm:4.2} and \Cref{thm:4.4}.  Roughly speaking, \Cref{thm:4.2} states that the set of polytopes $\Pf_v$ for $v\in H$ forms a polyhedral tiling on $\Omega_{r+1}$. Then, \Cref{thm:4.4} is deduced from \Cref{thm:4.2} and a result by Li \cite{10.1093/imrn/rnp201}, who computes the Chow classes of the irreducible components of $\LP(\V)$.

We believe that the techniques developed in this paper may also be useful for studying related spaces. For instance, inspired by work by Osserman and Esteves in \cite{Esteves2013}, we can define $\P(\V)$ to be the projection of $\LP(\V)$ onto the product of projective spaces associated to the extreme points of the minimum generating set $H$. This scheme could potentially be used to describe schematic limits of divisors in the Chow variety of the limit scheme, as an alternative to the Hilbert scheme. We plan to explore these connections in a future work.

%In short, here is a summary of the article. In \Cref{section:1} we collect some facts about polyhedral tilings. In \Cref{section:2} we introduce the necessary background on $\Z^n$-quivers and linked net. We introduce the notion of extreme points and discuss their connection to another similar notions in the literature.  In \Cref{section:3} we associate to a given linked net $\V$ a certain family of polytopes $\Pf_v$, one for each $v\in H$, where $H$ is the minimum set of generators of $\V$. We characterize the faces of $\Pf_v$ intersecting $\mathrm{\Omega}$. In \Cref{section:4} we show the main results, \Cref{thm:4.2} and \Cref{thm:4.4}.

As mentioned in the beginning, the results in this paper are motivated by \cite{Esteves22072025} and \cite{esteves2024quiverrepresentationsarisingdegenerations}. We briefly discuss how. 

Let $X$ be a connected reduced projective scheme over a field $\k$. A \textit{regular smoothing} of $X$ is the data of a flat projective map $\mathcal{X}\rightarrow B$, where $\mathcal{X}$ is regular and $B$ is the spectrum of a discrete valuation ring $R$ with residue field $\k$, and an isomorphism of the special fiber of the map with $X$. 

Let $(L_\eta,V_\eta)$ be a linear series on the general fiber of a regular smoothing $\mathcal{X}\rightarrow B$. Since $\mathcal{X}$ is regular, there is a line bundle extension  $\mathcal{L}$ of $L_\eta$ to $\mathcal{X}$. Let $X_0,\dots,X_n$ be the irreducible components of $X$; they are Cartier divisors of $\mathcal{X}$. Every other line bundle extension of $L_\eta$ is of the form $\mathcal{L}_u:=\mathcal{L}(\sum \ell_i X_i)$ for a unique $(n+1)$-tuple $u=(\ell_0,\dots,\ell_n)\in \Z^{n+1}_{\geq 0}$ with $\min\lbrace \ell_i\rbrace=0$. The set of those $(n+1)$-tuples is the vertex set of a quiver $Q$. As for the arrows, there is an arrow, and only one, connecting $u$ to $v$ if and only if $\mathcal{L}_u(X_i)\simeq \mathcal{L}_v$ for a certain $i$, which defines the type of the arrow. There is a representation $\mathfrak{L}$ associating to a vertex $u$ of $Q$ the line bundle $L_u:=\mathcal{L}_u\vert_X$, and to the arrows the restriction to $X$ of the maps $\mathcal{L}_u\rightarrow \mathcal{L}_u(X_i)$. Finally, there is a subrepresentation $\V$ of the representation obtained by taking global sections in $\mathfrak{L}$ which associates to a vertex $u$ of  $Q$ the image in $H^0(X,L_u)$ of $V_\eta\cap H^0(\mathcal{X},\mathcal{L}_u)$.

As mentioned before, $Q$ is a $\Z^n$-quiver,  and $\V$ is an exact finitely generated linked net. The triple $\mathfrak{g}=(Q,\mathfrak{L},\mathfrak{V})$ is called a \textit{ linked net of linear series}. In this case, $\LP(\V)$ is a flat degeneration of the small diagonal in $\prod_{v\in H}\P(V_v)$, where $H$ is the minimum set of generators of $\V$. 

The assignment taking $\mathfrak{M}\in \LP(\V)$ to $Z(\mathfrak{M})$, the intersection of the zero loci of the sections in $\mathfrak{M}$, is the underlying function of a scheme morphism $\LP(\V)\rightarrow \mathrm{Hilb}_X$. In addition, the image of $\LP(\V)$ in $\mathrm{Hilb}_X$ is the associated reduced subscheme of the limit of $\P(V_\eta)$ viewed naturally as a subscheme of the generic fiber of $\mathrm{Hilb}_{\mathcal{X}/B}$.

%% file: poly.tex
\section{Tilings induced by families of vector spaces.}\label{section:1}

In this section, we collect some facts about polyhedral tilings. The material in this section is based on \cite{amini2024tropicalizationlinearseriestilings}.

\subsection{Polytopes.} \label{subsection:1.1}

We review the theory of modular pairs and their base polytopes. % in the language of \cite{amini2024tropicalizationlinearseriestilings}. \\

Let $H$ be a finite nonempty set, and $\R^H$ the vector space whose elements are maps $q:H\rightarrow \R$. For each integer $d$, we denote by $\R^H_d$ the affine subspace of $\R^H$ consisting of all $q$ with $q(H):=\sum_{v\in H} q(v)=d$. For each $I\subseteq H$, set $q(I):=\sum_{v\in I} q(v)$. 

We denote by $2^H$ the family of subsets of $H$. For $I\subseteq H$, denote by $I^c$ the complement of $I$ in $H$, that is, $I^c=H-I$.  

We say a function $\mu:2^H\rightarrow \R$ is \textit{supermodular} if $\mu(\emptyset)=0$ and we have inequalities $$\mu(I_1)+\mu(I_2)\leq \mu(I_1\cup I_2)+\mu(I_1\cap I_2) \text{ for each } I_1,I_2\subseteq H.$$
 Similarly, we say that a function $\mu:2^H\rightarrow \R$ is \textit{submodular} if the above inequalities are all reversed. In any case, the quantity $\mu(H)$ is called the \textit{range} of $\mu$. Functions that are both supermodular and submodular are called \textit{modular}.

For each $\mu:2^H\rightarrow \R$, let $\mu^*:2^H\rightarrow \R$ be the \textit{function adjoint to} $\mu$, defined by $$\mu^*(I):=\mu(H)-\mu(H-I) \text{ for each } I\subseteq H.$$
Observe that $\mu$ is supermodular (resp. submodular) if and only if $\mu^*$ is submodular (resp. supermodular). When $\mu$ is supermodular, we refer to $\mu^*$ as the \textit{submodular function adjoint to }$\mu$ and call the ordered pair $(\mu,\mu^*)$ an \textit{adjoint modular pair}, or simply a \textit{modular pair}. 

For each modular pair $(\mu,\mu^*)$, we define the subset in $\R^H$:
$$\Pf_\mu=\Pf_{(\mu,\mu^*)}:=\left\lbrace q\in \R^H \,  \vert \, \mu(I)\leq q(I)\leq \mu^*(I) \text{ for each } I\subseteq H\right\rbrace.$$
We say that $\Pf_\mu$ is the \textit{base polytope associated to} $\mu$, or simply the \textit{polytope of} $\mu$.

An \textit{ordered partition} of $H$ is an ordered sequence $\pi=(\pi_1,...,\pi_s)$ of pairwise disjoint subsets of $H$ whose union is equal to $H$. It is called \textit{nontrivial} if $\pi_i\neq\emptyset$ for every $i$. If $\pi$ is nontrivial, we call it a \textit{bipartition} if $s=2$, and a $s$-\textit{partition} in general. If $\pi$ is a bipartition, we put $\pi^c=(\pi_2,\pi_1)$.

The data of an ordered partition $\pi$ as above is equivalent to the data of a filtration $$F_\bullet :=(F_0,\dots,F_s) \text{ where } \emptyset=F_0\subseteq F_1\subseteq F_2\subseteq \cdots \subseteq F_s=H$$ 
via the correspondence $F_j:=\pi_1\cup\cdots \cup \pi_j$ for $j=1,...,s$ in one direction, and $\pi_j:=F_j-F_{j-1}$ for $j=1,...,s$ in the other direction. 

For each supermodular function $\mu$ and each two subsets $J_1,J_2\subseteq H$ with $J_1\subseteq J_2$, we define
$$\mu_{J_2/J_1}(I):=\mu((I\cap J_2)\cup J_1)-\mu(J_1) \text{ for each } I\subseteq H$$
Then, $\mu_{J_2/J_1}$ is a supermodular function. We say that $\mu_{J_2/J_1}$ is \textit{obtained by restricting to $J_2$ and contracting $J_1$}. 

For each ordered partition $\pi=(\pi_1,\dots,\pi_s)$ of $H$ with $F_\bullet$ the corresponding filtration, we define the function $\mu_\pi$ on $2^H$ as the sum $$\mu_\pi=\sum_{i=1}^s\mu_{F_i/F_{i-1}}.$$
Note that $\mu_\pi$ is supermodular. We say that $\mu_\pi$ is the \textit{splitting of $\mu$ with respect to the ordered partition $\pi$}. We say that a supermodular function $\nu$ is a \textit{splitting of $\mu$} if it coincides with $\mu_\pi$ for some ordered partition $\pi$ of $H$ and call the splitting \textit{nontrivial} if $\nu\neq \mu$. 

%\textcolor{red}{colocar aqui algo que puedo necesitar mas adelante}

Let $\mu$ be a supermodular function. Denote by $\cd_\mu$ the largest integer $s$ for which there is a $s$-partition $\pi$ of $H$ such that $\mu=\mu_\pi$. We call $\mathrm{cd}_\mu$ the \textit{codimension} of $\mu$. We do so because $\dim \Pf_\mu+\mathrm{cd}_\mu=\vert H\vert $; see \cite[Prop.\,2.7]{amini2024tropicalizationlinearseriestilings}.

A supermodular function $\mu$ is called \textit{simple} if there is no nontrivial ordered partition $\pi$ of $H$ such that $\mu=\mu_\pi$, or equivalently, if $\Pf_\mu$ has dimension $\vert H\vert-1$, or equivalently, $\cd_\mu=1$. 

Let $\mu,\nu$ be two supermodular functions of range $n$, and let $\pi=(I,J)$ be a bipartition of $H$. We say that $\pi$ is a \textit{separation} for the ordered pair $(\mu,\nu)$ if $\mu(I)+\nu(J)\geq n$. We say that the separation is strict if the inequality is strict. We say that the separation is 
\textit{nontrivial} provided that it is either strict, or we have $\mu_\pi\neq \mu$ or $\nu_\pi\neq \nu$. Alternatively, $\pi$ is a separation for $(\mu,\nu)$ if and only if $\mu(I)\geq \nu^*(I)$, if and only if $n\geq \nu^*(I)+\mu^*(J)$. 
%\begin{prop}   Let $\mu$ be a supermodular function and $\pi=(\pi_1,\dots,\pi_s)$ an ordered partition of $V$. Let $F_\bullet$ be the corresponding filtration. Then, we have  $$\Pf_{\mu_\pi}=\Pf_\mu\cap \lbrace q\in \R^V \vert q(F_j)=\mu(F_j) \text{ for } j=1,\dots, s-1\rbrace$$   Moreover, the base polytope $\Pf_\mu$ is nonempty and we can recover $\mu$ (and $\mu^*$) from $\Pf_\mu$ as follows:    $$\mu(I):=\min\lbrace q(I) \vert q\in\Pf_\mu\rbrace \text{ and } \mu^*(I):=\max\lbrace q(I) \vert q\in\Pf_\mu\rbrace \text{ for each } I\subseteq V.$$ In particular,  $\Pf_{\mu_1}=\Pf_{\mu_2}$ if and only if $\mu_1=\mu_2$. In addiction, the base polytopes $\Pf_{\mu_\pi}$ for ordered partitions $\pi$ of $V$ are all the faces of $\Pf_\mu$. Furthermore, the vertices of $\Pf_\mu$ are the points $\Pf_{\mu_\pi}$ where $\pi=(\pi_1,\dots,\pi_s)$ is nontrivial and $s=\# V$. \end{prop}

The following lemma will be used in the subsequent sections. 
\begin{lem}\label{lem:1.1}
Let $\mu:2^H\rightarrow\R$ be a nonnegative supermodular function.  Then, both $\mu$ and $\mu^*$ are nondecreasing. In addition, if \begin{enumerate}
    \item \label{item:uno} $\mu^*(v)\neq 0$ for all $v\in H$, and
    \item \label{item:dos} there is $v_0\in H$ such that $\mu^*(v_0)=\mu^*(H)$,
\end{enumerate}
then $\mu$ is simple. 
\end{lem}
\begin{proof}
It follows from \cite[Prop.\,2.7]{amini2024tropicalizationlinearseriestilings} that for each ordered partition $\pi$ of $H$, we have $\Pf_{\mu_\pi}\neq\emptyset$.
Let $J\subseteq I$ be subsets of $H$. Since $\mu$ is nonnegative, for each $q\in \Pf_\mu$,
$$q(I-J)\geq \mu(I-J)\geq 0.$$
By \cite[Prop.\,2.5]{amini2024tropicalizationlinearseriestilings}, there exists $q\in \Pf_\mu$ such that $q(I)=\mu(I)$. Then 
$$\mu(J)\leq q(J)\leq q(J) +q(I-J)=\mu(I).$$
We conclude that $\mu$ is nondecreasing. Applying this to $H-I$ and $H-J$, we obtain
$$\mu^*(J)=\mu(H)-\mu(H-J)\leq \mu(H)-\mu(H-I)=q(I)=\mu^*(I).$$
Hence, also $\mu^*$ is nondecreasing.

Now, assume $(1)$ and $(2)$ hold. Suppose by contradiction that there is a bipartition $\pi=(\pi_1,\pi_2)$ such that $$\mu(H)=\mu(\pi_1)+\mu(\pi_2),$$ or equivalently $$\mu^*(H)=\mu^*(\pi_1)+\mu^*(\pi_2).$$ Since $\pi$ is a bipartition, either $v_0\in \pi_1$ or $v_0\in\pi_2$. Assume $v_0\in \pi_1$. As $\mu^*$ is nondecreasing, we have $$\mu^*(v_0)\leq \mu^*(\pi_1)\leq \mu^*(H).$$ Hence, $\mu^*(\pi_1)=\mu^*(H)$, and thus $\mu^*(\pi_2)=0$. This implies that $\mu^*(u)=0$ for all $u\in \pi_2$, contradicting (1). 
\end{proof}

\subsection{Families}\label{subsection: 1.2}
Let $\mathcal{C}$ be a family of supermodular functions on $H$ of range $d$.  We denote by $\Pf_\mathcal{C}$ the union of the polytopes $\Pf_\mu$ for $\mu\in\mathcal{C}$ and call it the \textit{support of the collection} $\mathcal{C}$. 
 Let $\Omega\subseteq \R^H_d$ be a subset. We say that $\mathcal{C}$ is \textit{complete for} $\Omega$ if for each element $\mu\in\mathcal{C}$ and each bipartition $\pi$ of $H$ such that $\Pf_{\mu_\pi}$ intersects $\Omega$, there is $\mu'\in \mathcal{C}$ distinct from $\mu_\pi$ such that $\mu'_{\pi^c}=\mu_\pi$. We say that $\mathcal{C}$ is separated if each pair $(\mu,\nu)$ of  distinct elements of $\mathcal{C}$ admits a nontrivial separation.

\subsection{Polytopes associated to vector spaces.}\label{subsection: 1.3}

 Let $\k$ be a field. For each $v\in H$, let $U_v$ be a finite-dimensional vector space over $\k$. Define $U:=\bigoplus_{v\in H} U_v$. For each subset $I\subseteq H$, let $U_I:=\bigoplus_{v\in I} U_v$, and denote by $\iota_I\colon U_I\rightarrow U$ and $\theta_I\colon U\rightarrow U_I$ the corresponding insertion and projection maps, respectively. The composition $\theta_I\circ\iota_I$ is the identity. By convention, we set $U_\emptyset:=(0)$. We have a natural exact sequence $$0\rightarrow U_I \xrightarrow{\iota_I} U\xrightarrow{\theta_{I^c}} U_{I^c}\rightarrow
0.$$ We view elements $\varphi\in\k^H$ as endomorphisms of $U$ which, by abuse of notation, we will also denote by $\varphi$. An endomorphism $\varphi\in \k^H$ is an automorphism if and only if $\varphi_v\neq0$ for every $v\in H$.

Let $W\subseteq U$ be a vector subspace. For each $I\subseteq H$, we put $W_I:=\iota_I\circ\theta_I(W)$ and denote by $\mu_W^*(I)$ its dimension, that is, $\mu_W^*(I):=\dim_\k(W_I)$. Similarly, we denote by $W^I$ the intersection of $W$ with $\iota_I(U_I)$, and by $\mu_W(I)$ its dimension. We have a short exact sequence 
$$0\rightarrow W^{I^c}\rightarrow W\rightarrow W_I\rightarrow 0$$
where the second map is the inclusion and the third map is the composition $\iota_I\circ\theta_I$. Furthermore, if $W$ has dimension $r+1$, then $(\mu_W,\mu_W^*)$ is a modular pair of range $r+1$. Put $\Pf_W:=\Pf_{\mu_W}$, the polytope of $\mu_W$. Denote by $\Omega_{r+1}$ the standard simplex in $\R_{r+1}^H$, given by $$\Omega_{r+1}:=\left\lbrace q\in \R_{r+1}^H\, \vert \, q(v)\geq 0 \text{ for all } v\in H \text{ and } q(H)=r+1\right\rbrace.$$
The polytope $\Pf_W$ is a subset of the standard simplex $\Omega_{r+1}$ in $\R_{r+1}^H$. Note as well that $\Pf_{W}=\Pf_{\varphi W}$ for all automorphisms $\varphi\in \k^H$.

\subsection{Splittings.} \label{subsection: 1.4} Let $\pi=(\pi_1,\dots,\pi_s)$ be a nontrivial ordered partition of $H$. Denote by $F_\bullet$ the corresponding filtration of $H$. Let $W\subseteq U$ be a vector subspace. We associate to $\pi$ a subspace $W_\pi$ of $U$ defined as follows: Consider the filtration $$0\subseteq W^{F_{1}}\subseteq W^{F_{2}}\subseteq\cdots\subseteq W^{F_{n-1}}\subseteq W,$$ 
and let $W_\pi$ be its associated graded object. We regard $W_\pi$ as a subspace of $U$ as follows: For each $j\in [s]$,  we have the following exact sequence:
$$0\rightarrow W^{F_{j-1}}\rightarrow W^{F_j}\rightarrow \theta_{\pi_j}(W^{F_j})\rightarrow0$$
Hence, the map $\theta_{\pi_j}\colon U\rightarrow U_{\pi_j}$ induces an isomorphism $W^{F_j}/W^{F_{j-1}}\simeq \theta_{\pi_j}(W^{F_j})$. We identify $W_\pi$ with $\bigoplus_{j=1}^s \theta_{\pi_j}(W^{F_j})$, which is a subspace of $U$ under the natural identification $U=\bigoplus_{j=1}^s U_{\pi_j}$. It follows from \cite[Prop\,5.2]{amini2024residuepolytopes} that $\mu_{W,\pi}=\mu_{W_\pi}$, i.e., the splitting of $\mu_W$ with respect to the ordered partition $\pi$ is equal to $\mu_{W_\pi}$. 

%Let $W_u,W_v\subseteq U$ be two vector subspaces of the same dimension. A morphism $c:W_v\rightarrow W_u$ is an element of $\k^H$. Assume that there is morphism $c:W_v\rightarrow W_u$. The \textit{support} of $c$ is the set of $w\in V$ such that $c_w\neq 0$.  \textcolor{red}{definir esto antes}

\subsection{Polyhedral tilings.}\label{subsection: 1.5}

For a polytope $\Pf$ in $\R^H_{r+1}$, the \textit{relative interior} of $\Pf$, denoted by $\mathring{\Pf
}$,  is the set of all points $q\in  \Pf$ that do not belong to any proper face of $\Pf$. 

Given a vector space $W\subseteq U$, the associated supermodular function $\mu_W$ is simple if and only if $\mathring{\Pf}_W$ is an open set in $\R^H_{r+1}$. In this case, we say that $W$ is \textit{simple}. 

 Let $W_1$ and $W_2$ be two vector subspaces of $U$ of the same dimension. We say that $W_1$ and $W_2$ are \textit{equivalent} if there exists $\varphi\in (\k^*)^H$ such that $\varphi W_1\subseteq  W_2$. In this case, $\Pf_{W_1}=\Pf_{W_2}$. We also say that $\Pf_{W_1}$ and $\Pf_{W_2}$ \textit{intersect in codimension at most} $1$ if there are faces $\F_1$ and $\F_2$ of $\Pf_{W_1}$ and $\Pf_{W_2}$ respectively of codimension at most $1$ such that $\mathring{\F}_1\cap \mathring{\F}_2\neq\emptyset$.

The following propositions are simplified versions of results by Amini and Esteves in a forthcoming paper. For the sake of completeness, their proofs will be sketched.

\begin{prop}\label{prop:1.2}
Let $W_1$ and $W_2$ be two vector subspaces of $U$ of the same dimension. Let $\varphi\in \k^H$ be a nonzero morphism such that $\varphi W_1\subseteq W_{2}$. Assume that $W_1$ and $W_2$  are simple. Then either $\varphi$ is an isomorphism or the relative interiors $\mathring{\Pf}_{W_1}$ and $\mathring{\Pf}_{W_2}$ are disjoint.    
\end{prop}
\begin{proof}
Set $I:=\lbrace v\in H\, \vert \, \varphi_v\neq 0\rbrace$. Since $\varphi\neq 0$, we have $I\neq\emptyset$. If $I=H$, then $\varphi$ is an isomorphism. 
Assume $I\neq H$. Then $\pi=(I,I^c)$ is a bipartition. The morphism $\varphi$ induces an injective morphism $c_{\varphi}\colon W_{1,I}\rightarrow W_{2}^I$, whence the inequality $\dim W_{1,I}\leq \dim W_{2}^I$, which implies that $\pi$ is a separation of $\mu_{W_2}$ and $\mu_{W_1}$. Since $W_1$ and $W_2$ are simple, the separation is nontrivial. Thus, by \cite[Prop.\,3.4]{amini2024tropicalizationlinearseriestilings}, the relative interiors $\mathring{\Pf}_{W_1}$ and $\mathring{\Pf}_{\W_2}$ are disjoint.    
\end{proof}
  Let $\mathcal{W}$ be a finite collection of simple vector subspaces of $U$ of dimension $r+1$. We say that $\mathcal{W}$ is \textit{complete} for a subset $\Omega\subseteq\R^H$ if for each $W\in \mathcal{W}$ and each bipartition $\pi$ of $H$ such that $\Pf_{W_\pi}$ intersects $\Omega$, there exists $W'\in\mathcal{W}$ distinct from $W$ such that $\mu_{W'_{\pi^c}}=\mu_{W_\pi}$ (and hence $\mu_{W'}\neq \mu_W$ because $W$ is simple), i.e., if the collection of associated supermodular functions is complete for $\Omega$.

\begin{prop}\label{prop:1.3}
     Let $\mathcal{W}$ be a finite collection of simple nonequivalent vector subspaces of $U$ of dimension $r+1$. Assume: \begin{enumerate}
        \item\label{C1} $\mathcal{W}$ is complete for $\mathring{\Omega}_{r+1}$ and $\Pf_{\mathcal{W}}$ intersects $\mathring{\Omega}_{r+1}$. 
      %  \item\label{C2} $\F\cap \mathring{\Omega}_{r+1}\neq\emptyset$ for each facet $\F$ of $\Pf_{W}$ for each $W\in \mathcal{W}$ such that $\mathring{\Omega}_{r+1}$ contains points of both sides of the span of $\F$. \textcolor{red}{me parece que no es necesario en este caso.}
        \item\label{C3}  For each $W_1,W_2\in\mathcal{W}$, %such that $\Pf_{W_u}$ and $\Pf_{W_v}$ intersect in codimension at most $1$%
         there is a nonzero $c\in \k^H$ such that $cW_1\subseteq W_2$. 
        \end{enumerate}
        Then the collection of polytopes $\Pf_{W_\pi}$ for $W\in\mathcal{W}$ and $\pi$ ordered partition of $H$ is a tiling of $\Omega_{r+1}$.  
\end{prop}

\begin{proof}
   It follows from \cite[Prop.\,3.8]{amini2024tropicalizationlinearseriestilings} that $\mathring{\Omega}_{r+1} \subseteq \Pf_\mathcal{W}$. Since $\Pf_{W_{\pi}}\subseteq \Omega_{r+1}$ for each $W\in\mathcal{W}$ and each ordered partition $\pi$ of $H$, we have $\Omega_{r+1}=\Pf_{\mathcal{W}}$. We need only show that for each two $W,W'\in \mathcal{W}$, the polytopes $\Pf_{W}$ and $\Pf_{W'}$ either are disjoint or intersect in a common face.
   Assume $\Pf_{W}\cap \Pf_{W'}\neq\emptyset$. We will show that $\Pf_{W}\cap \Pf_{W'}$ is a face of $\Pf_{W}$. Let $q\in \Pf_{W}\cap \Pf_{W'}$. Let $\F$ be the face of $\Pf_{W}$ that contains $q$ in its relative interior. Since $\Pf_{W}\cap \Pf_{W'}$ is convex and $q$ is arbitrary, it suffices to show that $\F\subseteq \Pf_{W'}$. Indeed,  it would follow that $\Pf_{W}\cap \Pf_{W'}$ is a union of faces, so, by convexity, a single face of $\Pf_{W}$. It follows from $(2)$ and \Cref{prop:1.2} that either $W$ and $W'$ are equivalent or the relative interiors $\mathring{\Pf}_{W}$ and $ \mathring{\Pf}_{W'}$ are disjoint. Since the $W$ in the collection $\mathcal{W}$ are nonequivalent, we have $\mathring{\Pf}_{W}\cap \mathring{\Pf}_{W'}=\emptyset$. 

   We claim the existence of a sequence $W_0,\dots,W_m\in \mathcal{W}$ such that $\mu_{W_0}=\mu_{W}$ and $\mu_{W_m}=\mu_{W'}$, and, for each $j=1,\dots,m$, $$(*) \hspace{1cm} \Pf_{W_{j-1}}\cap \Pf_{W_j} \text{ is a common facet of }\Pf_{W_{j-1}} \text{ and }  \Pf_{W_j} \text{ containing } \F.$$ 
   If such sequence exists, then the proposition follows, as $\F\subseteq\Pf_{W'}$. 

   Put $W_0=W$. Assume that we are given a sequence $W_1,\dots,W_i\in\mathcal{W}$ for some $i\geq 0$ such that $(*)$ holds for each $j=1,\dots,i$. Notice that $\F$ is a face of $\Pf_{W_j}$ for each $j$. If $\mu_{W'}=\mu_{W_i}$, the claim is proved. Suppose not. Then, since $W_i$ and $W'$ are nonequivalent, it follows from $(2)$ and \Cref{prop:1.2} that $\mathring{\Pf}_{W_i}\cap \mathring{\Pf}_{W'}=\emptyset$.  Hence, there exists a facet $\F'$ of $\Pf_{W_i}$ containing $\F$ such that $\mathring{\Pf}_{W_i}$ and $\mathring{\Pf}_{W'}$ lie on opposite sides of the hyperplane $H$ spanned by $\F'$. Since $\F'$ is a facet, $\F'=\Pf_{W_{i,\pi}}$ for a certain bipartition $\pi$ of $H$.

   We claim that $\F'\cap \mathring{\Omega}_{r+1}\neq\emptyset$. Assume by contradiction that $\F'\cap \mathring{\Omega}_{r+1}=\emptyset$. Then the hyperplane $H$ is cut out by an equation of the form $q(v)=0$ for a certain $v\in H$. Since $\Pf_{W_i}$ and $\Pf_{W'}$ are contained in $\Omega_{r+1}$, we have that either $\Pf_{W_i}$ or $\Pf_{W'}$ is contained in  $$\left\lbrace q\in \R^H_{r+1} \, \vert \, q(v)\leq 0\right\rbrace\cap \Omega_{r+1}.$$%=\Omega_{r+1}\cap \left\lbrace q(v)=0\right\rbrace.$$ 
   This contradict the hypothesis that the $W_i$ and $W'$ are simple. Hence, $\F'\cap\mathring{\Omega}_{r+1}\neq\emptyset$. 

   Since $\mathcal{W}$ is complete for $\mathring{\Omega}_{r+1}$, there is $W_{i+1}\in\mathcal{W}$ distinct from $W_i$ such that $\Pf_{W_{i+1,\pi^c}}=\F'$. As $W_i$ and $W_{i+1}$ is simple, $\mathring{\Pf}_{W_{i+1}}$ and $\mathring{\Pf}_{W_i}$ lie on opposite sides of the hyperplane $H$, thus $\Pf_{W_{i+1}}$ lies on the same side as $\Pf_{W'}$. Since the collection $\mathcal{W}$ is finite, and we ``approach'' $\Pf_{W'}$ at each step, there is $m$ such that $\mu_{W_m}=\mu_{W'}$.

\end{proof}

%% file: Znquiversandlinkednets.tex
\section{$\Z^{n}$-quivers and linked nets.}\label{section:2}

In this section, we collect some results about $\Z^n$-quivers and linked nets. We refer to \cite{Esteves22072025}  and \cite{esteves2024quiverrepresentationsarisingdegenerations}, which contain the necessary background. 

\subsection{$\Z^{n}$-quivers.} \label{subsection:2.1}

Let $Q$ be a quiver. Let $T$ be a nontrivial partition of the arrow set of $Q$. Each part $\a$ is called an \textit{arrow type}, and we say $a\in \a$ has type $\a$.
% The number of parts is $n+1$ with $n\in \Z_{\geq 0}$. 

Given a path $\gamma$ in $Q$ and an arrow type $\a$, we denote by $\t_{\gamma}(\a)$ the number of arrows of that type that the path contains. We call $\t_{\gamma}$ the \textit{type} of $\gamma$, and the collection of arrow types $\lbrace \a\vert \t_{\gamma}(\a)>0\rbrace$ its \textit{essential type}, denoted by $\ess(\gamma)$. The path $\gamma$ is called $\textit{admissible}$ if $\t_{\gamma}(\a)=0$ for some $\a$ and \textit{simple} if $\t_{\gamma}(\a)\leq 1$ for every $\a$. . 

A quiver $Q$ with a nontrivial partition $T$ of the set of arrows in $n+1$ parts is a $\Z^{n}$-\textit{quiver} if the following three conditions are satisfied: \begin{enumerate}
    \item There is exactly one arrow of each type leaving each vertex.
    \item Each vertex is connected to each other by an admissible path.
    \item Two paths $\gamma_1$ and $\gamma_2$ leaving the same vertex arrive at the same vertex if and only if $\t_{\gamma_1}-\t_{\gamma_2}$ is a constant function. 
\end{enumerate}
A simple non-admissible path is thus a circuit, called a \textit{minimal circuit}. Two distinct vertices connected by a simple path are called \textit{neighbors}. If $v_1$ and $v_2$ are neighbors and $I$ is the essential type of a simple path connecting $v_1$ to $v_2$ we write $v_2=I\cdot v_1$. 

For a vertex $v$ of a $\Z^n$-quiver $Q$ and a set $I$ of arrow types, the $I-$\textit{cone} of $v$, denoted by $C_{I}(v)$, is the set of end vertices of all admissible paths leaving $v$ and having essential type contained in $I$. 

Recall that the \textit{hull} of a collection of vertices $H$ of a $\Z^n$-quiver $Q$ is the set $P(H)$ of all vertices $v$ of $Q$ such that for each arrow type there are $z\in H$ and a path $\gamma$ connecting $z$ to $v$ not containing any arrow of that type.

For each vertex $v$ in $Q$ there is $w_v\in P(H)$ such that for each $z\in H$ there is an admissible path connecting $z$ to $v$ through $w_v$. Furthermore, $w_v$ is unique if $P(H)=H$. We call $w_v$ the \textit{shadow} of $v$ in $H$. See \cite[Prop.\,5.7]{Esteves22072025} for further details. 
\subsection{Extreme points.}\label{subsection:2.1}
\begin{lem}\label{lem:2.1}
    Let $\a$ be an arrow type of a $\Z^n$-quiver $Q$ and $H$ a finite nonempty collection of vertices. Put $I_\a:=T-\lbrace \a\rbrace$.  The following statements hold:\begin{enumerate}
        \item There exists a vertex $u_{\a}\in H$ such that $C_{I_\a}(u_{\a})\cap H=\lbrace u_\a\rbrace$. 
        \item If $P(H)=H$, then $u_{\a}$ is unique. Furthermore, every admissible path connecting each $u\in H$ to $u_{\a}$ avoids $\a$. 
   %     \item If $P(H)=H$, and $v\not\in H$ is a vertex such that for every $u\in H$, there is an admissible path from $u$ to $v$ avoiding $\a$, then there is an admissible path connecting $u$ to $v$ that passes through $u_{\a}$ %and avoids $\a$.
    \end{enumerate} 
\end{lem}
\begin{proof}
Set $I:=I_{\a}$. %  First, suppose that $H$ is finite and $\#H>1$ %
Since $H$ is finite, for each $v\in H$, the intersection $C_I(v)\cap H$ is nonempty and finite.  Define$$r(v):=\#(C_I(v)\cap H).$$ Consider Statement $(1)$. We aim to show that there exists $v\in H$ such that $r(v)=1$.  Suppose by contradiction that $r(v)\geq 2$ for all $v\in H$. Choose $v\in H$ such that $r(v)$ is minimum. Then there is $w\in C_I(v)\cap H$ such that  $w\neq v$. Note that $C_{I}(w)\subseteq C_I(v)$.  We claim that $v\not\in C_{I}(w)$. Suppose by contradiction that $v\in C_{I}(w)$. Then there is an admissible path connecting $w$ to $v$ avoiding $\a$. But a reverse path avoids $\a$ as well. This contradicts \cite[Prop.\,5.1]{Esteves22072025}. Therefore, $v\not\in C_I(w)$, which implies $r(w)<r(v)$, contradicting the minimality of $r(v)$.  It follows that there exists $u_\a$ such that $C_{I}\cap H=\lbrace u_\a\rbrace$. 
    
    As for $(2)$, suppose by contradiction that $P(H)=H$ and there are two $u,v\in H$ such that $r(u)=r(v)=1$. Observe that $P(\lbrace u,v\rbrace)\subseteq H$. If $P(\lbrace u,v\rbrace)=\lbrace u,v\rbrace$, then  $u$ and $v$ are neighbors, and so either $u\in C_I(v)$ or $v\in C_I(u)$, which is a contradiction.  We may assume there is $w\in P(\lbrace u,v\rbrace)$ distinct from $u$ and $v$. 
    
    Let $\gamma$ and $\gamma'$ be admissible paths leaving $u$ and $v$, respectively, and arriving at $w$. Let $K$ and $J$ be the essential type of $\gamma$ and $\gamma'$, respectively. Since $w\in P(\lbrace u,v\rbrace)$, the sets $K$ and $J$ are disjoint. Then $\a\not\in K$ or $\a\not\in J$, whence $w\in C_{I}(u)$ or $w\in C_I(v)$. This implies $r(u)\geq 2$ or $r(v)\geq 2$. % We may assume $\a\in K$. As in \cite[Def.8.5.]{esteves2022quiverrepresentationsarisingdegenerations}, there is a unique lexicographic sequence of intermediate points from $u$ to $w$, more precisely, there are unique vertices $u_0,u_1,...,u_p$ such that \begin{enumerate}     \item $u_0=u$ and $u_p=w$.    \item $u_i=K_i\cdot u_{i-1}$ for a proper collection $K_i$ of arrow types for each $i$.        \item $K_1\subseteq K_2\subseteq \cdots \subseteq K_p$  \end{enumerate} In particular, $K_p=K$. Denote by $\theta_i$ an admissible path leaving $u_{i-1}$ and arriving at $u_i$, so $\ess(\theta_i)=K_i$ and $\ess(\theta_i^{-1})=T-K_i$. Since $\a\in K$, there exists and index $l$ such that $\a\in K_{l}-K_{l-1}$, which implies $\a\not\in \ess(\theta_l^{-1})$. Consider the path $$\theta=\theta_l^{-1}\cdots \theta^{-1}_{p}\gamma',$$ which is an admissible path from $v$ to $v_{l-1}$. The essential type of $\theta$ is $$\ess(\theta)=\bigcup_{j=l}^{p} (T-K_j)\cup J=T-K_l.$$ Hence, $\a\not\in \ess(\theta)$, so $v_{k-1}\in C_{I}(v)$. On the other hand, by construction, $v_{k-1}\in C_{I}(v)\cap C_{I}(u)\cap H$.
    This contradicts the assumption that $r(u)=r(v)=1$. Therefore, the vertex $u_\a$ satisfying $C_{I_\a}(u_\a)\cap H=\lbrace u_\a\rbrace$ is unique. 

    We prove now the second statement in $(2)$, that is, given $w\in H$, any admissible path from $w$ to $u_\a$ has essential type not containing $\a$, or equivalently, $u_\a\in C_{I}(w)$. If $r(w)=1$, then $w=u_\a$. Assume $r(w)\geq 2$. Then there is $w_1\in C_{I}(w)\cap H$ such that $w_1\neq w$. As before, $ C_I(w_1)\cap H\subsetneq C_I(w)\cap H$, and consequently  $r(w_1)<r(w)$. By repeating this argument, we obtain a finite sequence of vertices $w_1,\dots,w_n\in H$  such that  $$C_{I}(w_n)\cap H\subsetneq C_{I}(w_{n-1})\cap H\subsetneq\cdots \subsetneq C_{I}(w_1)\cap H\subsetneq C_I(w) \cap H$$ and $$1=r(w_n)<r(w_{n-1})<\cdots <r(w_1)<r(w).$$ From the uniqueness of $u_\a$ established earlier, it follows that $w_n=u_\a$. Hence $u_\a\in C_I(w)$, completing the proof of $(2)$. \end{proof}

   % As for $(3)$, it is sufficient to observe that for each  $u\in H$, there exists an admissible path connecting $u$ to $u_\a$ whose essential type does not contains $\a$, by part $(2)$. By assumption, there also exists an admissible path from $u_\a$ to $v$ that avoids $\a$. Therefore, the concatenation of this two paths yields an admissible path from $u$ to $v$ passing through $u_a$ whose essential type avoids $\a$, as required. 

\begin{defn}\label{def:2.2}
Let $Q$ be a $\Z^n$-quiver and $H$ a finite nonempty collection of vertices of $Q$ such that $P(H)=H$. For each arrow type $\a\in T$, the vertex $u_{\a}\in H$ is called the  \textit{extreme vertex} of $H$ with respect to $\a$.  Denote by $\e_H:T\rightarrow H$ the function taking an arrow type $\a$ to the corresponding vertex $u_{\a}\in H$.
\end{defn}

\begin{rem}\label{rem:2.3}
 The notion of extreme vertices turns out to be very closely related to the theory of $v$-reduced divisors on graphs. Let $G$ be a finite, unweighted multigraph having no loop edges. Denote $V(G)$ and $E(G)$, respectively, the set of vertices and edges of $G$. 
 Let $\mathrm{Div}(G)$ be the free Abelian group on $V(G)$. An element of $\mathrm{Div}(G)$ is called a \textit{divisor} on $G$. 

 We define an equivalence relation $\sim$ on the group $\mathrm{Div}(G)$ by declaring $D\sim D'$ if and only if there is a sequence of moves taking $D$ to $D'$ in the chip-firing game. For $D\in \mathrm{Div}(G)$, we define the \textit{linear system associated to} $D$ to be the set $\vert D\vert$ of all effective divisors equivalent to $D$: $$\vert D\vert=\lbrace D'\in\mathrm{Div}(G) \,\big\vert\, D'\geq 0, D'\sim D\rbrace.$$
Let $v\in V(G)$. We say that a divisor $D\in \mathrm{Div}(G)$ is $v$-\textit{reduced} if and only if $(1)$ no vertex $u\neq v$ is in debt; and $(2)$ for every non-empty subset $A$ of $V(G)-\lbrace v\rbrace$, if all vertices in $A$ were to perform a lending move, some vertex in $A$ would go into debt, see \cite{BAKER2007766}. 

The linear system $\vert D\vert$ can be made into a $\Z^n$-quiver as follows: The set of vertices is $\vert D\vert$, and we put an arrow from vertex $D_1$ to vertex $D_2$ if $D_2$ can be obtained from $D_1$ by a chip-firing move. The set of arrow types is identified with $V(G)$.

Let $H$ be the set of effective divisors in $\vert D\vert$. Then $H$ is finite and satisfies $P(H)=H$. Furthermore, a divisor in $\vert D\vert$ is $v$-reduced if and only if it is extreme with respect to $v$. 
\end{rem}

\subsection{Linked nets}\label{subsection:2.3}  
Let $\k$ be a field and call its elements \textit{scalars}. For each representation $\V$ of a quiver $Q$ in a $\k$-Abelian category,  denote by $V_v^{\V}$ the object associated to a  vertex $v$ of $Q$, and  denote by $\varphi^{\V}_{\gamma}$ the composition of morphisms associated to a path $\gamma$ in $Q$. (If $\gamma$ is the trivial path, $\varphi^\V_\gamma=\mathrm{id}$ by convention.) We write $V_v$ and $\varphi_\gamma$ if $\V$ is clear from the context. A representation of $Q$ is called 
\textit{pure} if every epimorphism between two objects among the $V_v$ is an isomorphism. 

Let $H$ be a set of vertices in $Q$ and $m\in\mathbb{N}\cup\lbrace\infty\rbrace$. We say $\V$ is $m$-\textit{generated} by $H$  if for each vertex $v$ of $Q$ there exist $w_1,\dots,w_l\in H$ for $l\in 
\mathbb{N}$ with $l\leq m$  and paths $\gamma_i$ connecting $w_i$ to $v$ for each $i$, such that $$(\varphi_{\gamma_1},\dots,\varphi_{\gamma_{m}}):\bigoplus^m_{i=1} V_{w_i}\longrightarrow V_v$$ is an epimorphism. If $\V$ is $\infty$-generated by a finite set, we say $\V$ is \textit{finitely generated}. Then $\V$ is $1$-generated by a finite set by \cite[Prop.\,6.4]{Esteves22072025}. 

If $Q$ is a $\Z^n$-quiver, we say $\V$ is a \textit{weakly linked net} over $Q$ if $\V$ satisfies the following two conditions: \begin{enumerate}
    \item If $\gamma_1$ and $\gamma_2$ are two paths connecting the same two vertices and $\gamma_2$ is admissible then $\varphi_{\gamma_1}$ is a scalar multiple of $\varphi_{\gamma_2}$;
    \item $\varphi_{\gamma}=0$ for each minimal circuit $\gamma$.
\end{enumerate}    
We say that $\V$ is a \textit{linked net} if in addition a third condition is verified: 
\begin{enumerate} 
    \item[(3)] If $\gamma_1$ and $\gamma_2$ are two admissible paths leaving the same vertex without arrow type in common, then $\ker(\varphi_{\gamma_1})\cap \ker(\varphi_{\gamma_2})=0.$ 
\end{enumerate}
For each vector $w$ in a vector space, we will denote by $[w]$ the set of its scalar multiples. And for each linear map $\varphi$, we denote $\ker[\varphi]:=\ker(\varphi)$ and $\im[\varphi]:=\im(\varphi)$.

Given two vertices $u_1$ and $u_2$ of $Q$, let $\varphi^{u_1}_{u_2}:=[\varphi_{\gamma}]=\k\varphi_\gamma$ for any admissible path $\gamma$ connecting $u_1$ to $u_2$. We say  $\V$ is \textit{exact} if $\im(\varphi^{u_1}_{u_2})=\ker(\varphi^{u_2}_{u_1})$ for each two neighbors $u_1$ and $u_2$. 

\begin{rem}
   Let $\V$ be an exact linked net of vector spaces over a $\Z^n$-quiver $Q$. Then, exactness together with the rank-nullity theorem imply that $\V$ is pure.
\end{rem}

The following lemma extends \cite[Prop.\,10.7]{esteves2024quiverrepresentationsarisingdegenerations}.
\begin{lem}\label{lem:2.5}
    Let $\V$ be a nontrivial finitely generated exact linked net of vector spaces over $\k$. Let $H$ be the intersection of all
collections of vertices $1$-generating $\V$. Then $\V$ is $1$-generated by $H$.
Furthermore, %if $\k$ is infinite, then
for each $v,u\in H$, %there is a $s_v\in V_v$ such that and $\a\in T$
we have $\varphi^v_u\neq 0$. Also, $P(H)=H$.  
\end{lem}
\begin{proof} 
That $\V$ is $1$-generated by $H$ follows from \cite[Prop.\,6.3]{Esteves22072025}. %Assume that $\k$ is infinite. 
Set $u_\a:=\e_{P(H)}(\a)$ for each $a\in T$. 

As for the second statement, let $v\in H$ and $\a\in T$, %the subspace $\ker(\varphi^v_{u_\a})\subseteq V_v$ is proper. Indeed, 
assume by contradiction that $\varphi^v_{u_\a}=0$. %for some $\a\in T$.
Let $\gamma$ be an admissible path connecting $v$ to $u_\a$ and $J$ be its essential type. Then $V_v=\ker(\varphi^v_{J\cdot v})$ by \cite[~Lem.\,6.6]{Esteves22072025}. Let $w$ be the shadow of $J\cdot v$ in $P(H)$. Then $w$ is $v$ or neighbor of $v$ and $\varphi^v_{J\cdot v}=\varphi^w_{J\cdot v}\varphi^v_w$. Since $\V$ is $1$-generated by $H$, it follows from \cite[Prop.\,6.4]{Esteves22072025} that $\V$ is $1$-generated by $P(H)$. Notice that $\V$ is pure, since it is exact. Hence, $\varphi^{w}_{J\cdot v}$ is an isomorphism. Then $w\neq v$ and $\ker(\varphi^v_{J\cdot v})=\ker(\varphi^v_w)$. It follows from exactness of $\V$ that  $$V_v=\ker(\varphi^v_{u_\a})=\ker(\varphi^v_{J\cdot v})=\ker(\varphi^v_{w})=\text{Im}(\varphi^w_v).$$ Put $H':=(H-\lbrace v\rbrace)\cup\lbrace w\rbrace$. We claim that $\V$ is $1$-generated by $H'$. Indeed, let $u$ be a vertex of $Q$. By the first statement, there is $z\in H$ such that $\varphi^z_u$ is an isomorphism. If $z\neq v$, we are done. If $z=v$, then the composition $\varphi^v_u\varphi^w_v$ is an isomorphism. Hence, by \cite[Lem 6.2]{Esteves22072025}, $\varphi^v_u\varphi^w_v=\varphi^w_u$, whence $\varphi^w_u$ is also an isomorphism, proving the claim. But then $H\subseteq H'$ by definition of $H$, a contradiction.

%Since $\k$ is infinite and $\ker(\varphi^v_{u_\a})$ is a proper subspace of $V_v$ for each $\a\in T$, the avoidance lemma for vector spaces implies that there is $s_v\in V-\bigcup_{\a\in T} \ker(\varphi^v_{u_\a})$. Thus, the second statement is proved. 

We will show first that $P(H)=H$. Assume by contradiction that there exists $z\in P(H)-H$. We claim that there is an arrow type $\mathfrak{b}\in T$ such that $\varphi^z_{u_\mathfrak{b}}=0$. Indeed, since $\V$ is pure, it follows from the first statement that there exists $w\in H$ such that $\varphi^w_z$ is an isomorphism. Let $\gamma$ be an admissible path connecting $w$ to $z$ and $\mathfrak{b}\in \ess(\gamma)$. Consider an admissible path $\mu$ connecting $z$ to $u_\mathfrak{b}$. It follows from \Cref{lem:2.1} that $\mu\gamma$ is not admissible. Hence, by \cite[Lem. 6.1]{Esteves22072025}, $\varphi^z_{u_\mathfrak{b}}\varphi^w_z=0$. Thus, $$\varphi^z_{u_\mathfrak{b}}(V_z)=\varphi^z_{u_\mathfrak{b}}\varphi^w_z(V_w)=0,$$  
whence $\varphi^z_{u_\mathfrak{b}}=0$ as required. 

Let $v\in H$. Then $\varphi^{v}_{u_\mathfrak{b}}\neq 0$ as proved above. Hence, there is no admissible path connecting $v$ to $u_\mathfrak{b}$ through $z$. Since $u_\mathfrak{b}$ is extreme for $P(H)$, every admissible path connecting $z$ to $u_\mathfrak{b}$ avoids $\mathfrak{b}$. Thus, every admissible path connecting $v$ to $z$ contains an arrow of type $\mathfrak{b}$. Since $v$ is arbitrary, it follows that $z\not\in P(H)$, a contradiction, showing that $P(H)=H$. 

 Now, assume by contradiction that $\varphi^v_u=0$ for certain $u,v\in H$. Since $\V$ is nontrivial, we have that $u\neq v$. Let $\gamma$ be an admissible path connecting $v$ to $u$, and let $\a\not\in\ess(\gamma)$. Let $\gamma'$ be an admissible path connecting $u$ to $e_H(\a)$. Then $\gamma'\gamma$ is admissible by \Cref{lem:2.1}, and
    $$\varphi^v_{e_H(\a)}=\varphi^u_{e_H(\a)}\varphi^v_u=0,$$
    contradicting what we proved above.

\end{proof}

\begin{defn}\label{def:2.6}
    Let $\V$ be an exact finitely generated linked net of vector spaces. Let $H$ be the intersection of all collections of vertices $1$-generating $\V$. We call $H$ the \textit{minimum set of generators} of $\V$. 
\end{defn}

%% file: linkednetsvectorspaces.tex
\section{Linked nets and Polytopes.} \label{section:3}

\subsection{Adjoint modular pairs associated to linked nets.}\label{subsection:3.1}

Let $\V$ be a nontrivial finitely generated exact linked net of vector spaces over a $\Z^n$-quiver $Q$, and let $H$ be its minimum set of generators.

In the following, we will choose, for each $v,u\in H$, an admissible path $\gamma^v_u$ connecting $v$ to $u$. When the starting (ending) point is clear we write $\gamma_u$ (resp. $\gamma^v$).

For each $v\in H$ and for each nonempty subset $I\subseteq H$, define \begin{equation}
    \Psi^v_I:V_v   \longrightarrow   U_I:=\bigoplus_{u\in I} V_u \tag{3.1}\label{eq:3.1}
\end{equation} as $ \Psi^v_I(s):=(\varphi_{\gamma_u}(s) \vert u\in I)$. Set $U:=U_H$.
Let $W_v\subseteq U$ be the image of $\Psi^v_H$, that is, $$W_v:=\left \lbrace\Psi^v_H(s) \,\vert\, s\in V_v\right\rbrace.$$
 The maps $\Psi^v_I$ depends on the choice of an admissible path $\gamma_u$ connecting $v$ to $u$ for each $u\in I$. Different choices produce different subspaces $W_v$. However, since $\V$ is a linked net, all the $W_v$ obtained via this process lie in the same orbit by the natural action of $\prod_{v\in H} \k^\times$ defined in \Cref{subsection: 1.3}.

For each $v\in H$, we denote by $(\mu_v^{\V},\mu^{*,\V}_v)$ the modular pair corresponding to $W_v$, defined in \cref{subsection: 1.3}, and denote $\Pf_{\V,v}$ the corresponding base polytope. If $\V$ is clear from the context, we simply write $(\mu_v,\mu^{*}_v)$ and $\Pf_v$. The pair $(\mu_v,\mu^{*}_v)$ does not depend on the choice of admissible paths connecting $v$ to $u\in H$. In fact, for each $I\subseteq H$, 
$$\mu_v(I)=\dim\left(\bigcap_{u\in H-I} \ker(\varphi^v_u)\right) \text{ and } \mu^{*}_v(I)=\dim \left(V_v/\bigcap_{u\in I} \ker(\varphi^v_u)\right)$$

The following statement follows directly from \Cref{lem:1.1} and \Cref{lem:2.5}.
\begin{prop} \label{prop:3.1}
    Let $\V$ be a nontrivial finitely generated exact linked net of vector spaces, and let $H$ be its minimum set of generators. For each $v\in H$, the supermodular function $\mu_v$ is simple. 
\end{prop}

%\begin{proof} Assume, by contradiction, that there is a bipartition  $\pi=(I,J)$ such that $$\mu_v(H)=\mu_v(J)+\mu_v(I)$$ or equivalently $$\mu_v^*(H)=\mu_v^*(J)+\mu_v^*(I).$$     Since $\pi$ is a bipartition, either $v\in I$ or $v\in J$. Assume $v\in I$. Then $\mu_v^*(I)=\mu_v^*(H)$ and $\mu_v^*(J)=0$. Hence, for each $u\in J$, it follows that $Im(\varphi^v_u)=0$ contradicting the minimality of $H$.   \end{proof}
   
\subsection{Face structure.}\label{subsection:3.2}
Let $\V$ be a nontrivial finitely generated exact linked net of vector  spaces of dimension $r+1$ over a $\Z^n$-quiver $Q$, and let $H$ be its minimum set of generators. Recall that $P(H)=H$. For each $v\in H$, let $\Pf_v$ be the polytope of the modular pair $(\mu_v,\mu_v^*)$. In this section, we give a characterization of the faces of $\Pf_v$ that intersect $\mathring{\Omega}_{r+1}$, where $\Omega_{r+1}$ is the standard simplex in $\R^H_{r+1}$, as defined in \Cref{subsection: 1.3}.
\subsubsection{Polygons.}\label{subsubsection:3.2.1}
    Let $H$ be a set of vertices of $Q$ such that $P(H)=H$. Recall that, for each vertex $u$ of $Q$, there is a unique vertex $w_u\in H$ such that for every $v\in H$ there is an admissible path connecting $v$ to $u$ factoring through $w_u$ \cite[Prop.\,5.7]{Esteves22072025}. %If $z\in H$, set $w_z:=z$.
    We say that $w_u$ is the \textit{shadow} of $u$ in $H$. For each $v\in H$, let $R_v^H$ be the set of vertices of $Q$ whose shadow in $H$ is $v$; we write $R_v$ when there is no ambiguity. We call $R_v^H$ the \textit{shadow region of $v$ in $H$}. The collection of all shadow regions form a partition of the set of vertices of $Q$ which we call the \textit{shadow partition of} $H$. 
  
  Recall that a \textit{polygon} is a collection of pairwise neighboring vertices of $Q$. A polygon with $m$ vertices is called a $m$-$\textit{gon}$. Given a vertex $v$ of $Q$, a way of constructing a $m$-gon containing $v$ is to pick a nontrivial ordered $m$-partition $T=I_1\cup\cdots\cup I_m$ of the set of arrow types $T$ of $Q$, and put $v_1:=v$ and $v_{i+1}:=I_i\cdot v_{i}$ for $i=1,\dots,m-1$. Then $\lbrace v_1,\dots,v_m\rbrace$ is a $m$-gon We say that $v_1,\dots,v_m$ form an \textit{oriented $m$-gon}.  
  
  Let $\Delta$ be a $m$-gon in $Q$ and $v\in \Delta$. There is a unique ordering $v_1,\dots,v_m$ of the vertices of $\Delta$ such that $v=v_1$ and $v_1,\dots,v_m$ form an oriented $m$-gon  \cite[Prop.\,5.9]{Esteves22072025}. Recall that for each polygon $\Delta$ in $Q$, we have $P(\Delta)=\Delta$ \cite[Prop.\,5.10]{Esteves22072025}. Therefore, the shadow partition of $\Delta$ induces a partition $\pi$ of $H$ whose parts are the $ R_{v}^\Delta\cap H$ for $v\in \Delta$. Order the partition $\pi$ according to the associated oriented $m$-gon, i.e., put $\pi_i:=R_{v_i}^\Delta\cap H$. We call $\pi$ the \textit{ordered partition induced by $\Delta$ and $v$}. If $\Delta\subseteq H$, then $v_i\in \pi_i$ for each $i$. Hence, the ordered partition $\pi$ is a $m$-partition. 

The following lemmas are straightforward consequences of the definitions. Since they will be used repeatedly in the sequel, we state and prove them here for reference.
\begin{lem}\label{lem:3.2}
    Let $v_1,\dots,v_m$ be vertices of $Q$ forming an oriented $m$-gon $\Delta$. Put $v_{m+1}:=v$, and let $i\in \lbrace 1,\dots,m\rbrace$ and $u\in H-R_{v_i}^\Delta$. Then there exists an admissible path connecting $v_i$ to  $u$ passing through $v_{i+1}$. 
\end{lem}
\begin{proof}
 Let $v_j$ be the shadow of $u$ in $H$. It is sufficient to show that there is an admissible path from $v_i$ to $v_j$ factoring through $v_{i+1}$. Since $v_1,\dots,v_m$ form an oriented $m$-gon, there exists a $m$-partition $T=I_1\cup\cdots \cup I_m$ of set of arrows $T$ of $Q$ such that $v_{i+1}=I_i\cdot v_i$ for $i=1,\dots, m$. We need only show that $\ess(\gamma^{v_i}_{v_j})$ contains $I_i$. 
By assumption, either $j>i$ or $j<i$. If $j>i$, then $\ess(\gamma^{v_i}_{v_j})=I_{i}\cup\cdots \cup I_{j-1}$. If $j<i$, then $\ess(\gamma^{v_i}_{v_j})=I_{i}\cup\cdots\cup I_m\cup I_1\cup\cdots\cup I_{j-1}$. In any case, $\ess(\gamma^{v_i}_{v_j})$ contains $I_i$, as required.

    \end{proof}

\begin{lem}\label{lem:3.3}
    Let $\V$ be a nontrivial exact finitely generated linked net over a $\Z^n$-quiver $Q$. Let $v_1,\dots, v_m$ form an oriented $m$-gon. Then the sequence
    $$0\rightarrow \ker(\varphi^{v_1}_{v_i})\rightarrow \ker(\varphi^{v_1}_{v_{i+1}})\rightarrow \ker(\varphi^{v_i}_{v_{i+1}})\rightarrow 0 $$
    is exact, where the second map is the inclusion and the third map is $\varphi^{v_1}_{v_{i}}.$
\end{lem}
\begin{proof}
We need only show that the third map is surjective. Since $\V$ is exact, we have that $$\varphi^{v_1}_{v_{i}}( \ker(\varphi^{v_1}_{v_{i+1}}))= \varphi^{v_1}_{v_{i}}( \mathrm{Im}(\varphi^{v_{i+1}}_{v_1}))=\mathrm{Im}(\varphi^{v_{i+1}}_{v_{i}})=\ker(\varphi^{v_i}_{v_{i+1}}).$$
\end{proof}

  %%\begin{rem}\textcolor{red}{rever!}  Let $H$ be a set of vertices of $Q$ and $\mathcal{P}_H$ be the collection  of polygons $\Delta\subseteq H$. Note that $\mathcal{P}_H$ is an (abstract simplicial) flag complex: Indeed, for each $\Delta\in \mathcal{P}_H$ and $\Delta'\subseteq \Delta$, we have that $\Delta'$ is also a polygon. The flag condition follows from the definition of polygon. It is not a matroid, since it does not satisfies the augmentation property. The set of vertices of $\mathcal{P}_H$ is $H$.  \end{rem}

\subsubsection{Faces induced by polygons.}\label{subsubsection:3.2.2}
  Let $\V$ be an exact finitely generated nontrivial linked net of vector spaces of dimension $r+1$ over a $\Z^n$-quiver $Q$, and let $H$ be its minimum set of generators.  Let $\Delta\subseteq H$ be a polygon. We associate to $\Delta$ a modular pair $(\mu_\Delta,\mu_\Delta^*)$ as follows. 
  First, let $V_\Delta$ be the following direct sum of vector spaces:
  $$V_\Delta:=\bigoplus_{v\in \Delta} \left(\bigcap_{w\in H-R_v^\Delta}\ker(\varphi^v_w)\right).$$
  Next, define the linear map $\Psi^\Delta_H:V_\Delta\rightarrow U$ sending $(s_v\, \vert\, v\in \Delta)$ to  $( \varphi^{w_u}_u(s_{w_u}) \,\vert\, u\in H)$, where $w_u$ is the shadow of $u$ in $\Delta$. Let $W_\Delta:=\Psi^\Delta_H(V_\Delta)$, the image of $\Psi^\Delta_H$, and $(\mu_\Delta,\mu_\Delta^*)$ the associated modular pair, as defined in \Cref{subsection: 1.3}. Let $\Pf_\Delta$ be the base polytope of $\mu_\Delta$.

\begin{lem}\label{lem:3.5} Let $H$ be the minimum set of generators of $\V$. Let $\Delta\subseteq H$ be a $m$-gon and $v\in \Delta$. Let $\pi$ be the $m$-partition of $H$ induced by $\Delta$ and $v$, and let $W_{v,\pi}$ be the splitting of $W_v$ associated to $\pi$, and $\Pf_{v,\pi}$ the corresponding face of $\Pf_v$. The following statements hold:\begin{enumerate} 
    \item $\mathrm{cd}_{\mu_{\Delta}}=m;$
    
         \item $W_{v,\pi}$ and $W_\Delta$ are equivalent; in particular, $\Pf_{v,\pi}=\Pf_\Delta$;
        \item  $\mathring{\Pf}_\Delta\cap \mathring{\Omega}_{r+1}\neq \emptyset$.
    \end{enumerate}
\end{lem}
\begin{proof}
Let $v_1,\dots,v_m$ be the vertices of $\Delta$ ordered such that $v_1=v$ and $v_1,\dots,v_m$ form an oriented $m$-gon. Put $v_{m+1}:=v_1$. Then, $$V_\Delta=\bigoplus_{i=1}^m\left(\bigcap_{w\in H-\pi_i}\ker(\varphi^{v_i}_w)\right),$$
where $\pi_i=R^\Delta_{v_i}\cap H$ for each $i=1,\dots,m$. Since $v_{i+1}\in H-\pi_i$, it follows from \Cref{lem:3.2} and \cite[Lem.\,8.1]{Esteves22072025} that $$\bigcap_{w\in H-\pi_i}\ker(\varphi^{v_i}_w)=\ker(\varphi^{v_i}_{v_{i+1}}).$$
By exactness, we have that
$$V_\Delta=\bigoplus_{i=1}^m \ker(\varphi^{v_i}_{v_{i+1}})=\bigoplus_{i=1}^m\mathrm{Im}(\varphi^{v_{i+1}}_{v_i}).$$
It follows that $$W_\Delta=\bigoplus_{i=1}^m\Psi^{v_i}_{\pi_i}( \mathrm{Im}(\varphi^{v_{i+1}}_{v_i}))$$
where $\Psi^{v_i}_{\pi_i}$ is the map defined in \eqref{eq:3.1}. We deduce from \Cref{lem:1.1} and \Cref{lem:2.5} that, for each $i=1,\dots, m$, the subspace $\Psi^{v_i}_{\pi_i}( \mathrm{Im}(\varphi^{v_{i+1}}_{v_i}))\subseteq U_{\pi_i}$ is simple, whence $\mathrm{cd}_{\mu_\Delta}=m$, finishing the proof of $(1)$. 

 As for $(2)$, let $F_\bullet$ be the filtration of $H$ corresponding to $\pi$, and let $$0\subseteq W^{F_1}_v\subseteq\cdots \subseteq W^{F_{m-1}}_v\subseteq W_v$$
be the corresponding filtration of $W_v$. It is sufficient to show that $\Psi^{v_i}_{\pi_i}(\ker(\varphi^{v_i}_{v_{i+1}}))$ and $\theta_{\pi_i}(W_v^{F_i})$ are equivalent in $U_{\pi_i}$. Recall that each $W^{F_i}_v$ is the image of $\bigcap_{z\in H-F_i}\ker(\varphi^v_z)$ in $U$ via $\Psi^v_H$. Since $v_{i+1}\in H-F_i$, it follows from \Cref{lem:3.2} and \cite[Lem.\,8.1]{Esteves22072025} that $$\bigcap_{z\in H-F_i}\ker(\varphi^v_z)=\ker(\varphi^v_{v_{i+1}}).$$
Thus,
$$\theta_{\pi_j}(W_v^{F_j})=\Psi^v_{\pi_i}(\ker(\varphi^v_{v_{i+1}})).$$
For each $w\in\pi_i$, the concatenation $\gamma^{v_i}_w\gamma^v_{v_i}$ is admissible. Since $\gamma^v_w$ is admissible, by \Cref{lem:2.5}, there is $c_w\in \k^*$ such that $$\varphi_{\gamma^{v}_w}=c_w\varphi_{\gamma^{v_i}_w}\varphi_{\gamma^v_{v_i}}.$$
Set $c:=(c_w)\in (\k^*)^{\pi_i}$. Then $c\Psi^{v_i}_{\pi_i}\circ\varphi^{v}_{v_i}=\Psi^{v}_{\pi_i}$, whence
$\Psi^{v}_{\pi_i}(\ker(\varphi^v_{v_{i+1}}))$ and $\Psi^{v_i}_{\pi_i}\circ\varphi^{v}_{v_i}(\ker(\varphi^v_{v_{i+1}}))$ are equivalent. It follows from \Cref{lem:3.3} that 
$$\Psi^{v_i}_{\pi_i}\circ\varphi^{v}_{v_i}(\ker(\varphi^v_{v_{i+1}}))=\Psi^{v_i}_{\pi_i}(\ker(\varphi^{v_i}_{v_{i+1}}))$$
as required.

  As for $(3)$, suppose by contradiction that $\mathring{\Pf}_\Delta\cap \mathring{\Omega}_{r+1}=\emptyset$. Since $\Pf_v\subseteq\Omega_{r+1}$, it follows from \cite[Prop.\,2.7]{amini2024tropicalizationlinearseriestilings} that $\mu_{v,\pi}^*(I)=0$ for some nonempty $I\subsetneq H$. Hence, for each $u\in I$ we have $\mu_{v,\pi}^*(u)=0$, because of \Cref{lem:1.1}. Fix $u\in I$ and let $v_i$ be its the shadow in $\Delta$. Then $u\in \pi_i$. Now, we have  $$\mu_{v,\pi}^*(u)=\dim\varphi^v_u(\ker(\varphi^v_{v_{i+1}}))=\dim\varphi^{v_i}_u\varphi^v_{v_i}(\ker(\varphi^v_{v_{i+1}})).$$
  By \Cref{lem:3.3} and exactness, 
   $$ \mu_{v,\pi}^*(u)=\dim(\varphi^{v_i}_u(\ker(\varphi^{v_i}_{v_{i+1}}))=\dim(\varphi^{v_i}_u(\mathrm{Im}(\varphi^{v_{i+1}}_{v_{i}}))=\dim(\mathrm{Im}(\varphi^{v_{i+1}}_{u})).$$
  Hence, $\varphi^{v_{i+1}}_{u}=0$, contradicting the minimality of $H$.
\end{proof}

    \begin{prop}\label{prop:3.6}
    Let $\V$ be a nontrivial exact finitely generated linked net of vector spaces over a $\Z^n$-quiver $Q$, and $H$ its minimum set of generators. Then, for each $v\in H$ and each $m$-partition $\pi=(\pi_1,\dots,\pi_m)$ of $H$ such that $\Pf_{v,\pi}\cap \mathring{\Omega}_{r+1}\neq\emptyset$, there is a $m-$gon $\Delta\subseteq H$ such that $v\in \Delta$ and the ordered partition of $H$ induced by $\Delta$ and $v$ coincides with $\pi$. In particular, $\Pf_{v,\pi}=\Pf_\Delta$. 
\end{prop}

\begin{proof}
    %It follows from \cite[prop 2.7]{amini2024tropicalizationlinearseriestilings}  that there is
   
  Let $\F:=\Pf_{v,\pi}$ be the face of $\Pf_v$ corresponding to the $m$-partition $\pi$. We divide the proof in two steps. First, we will construct a certain oriented $m$-gon $\Delta$ containing $v$, and then we will show that the ordered partition induced by $\Delta$ and $v$ coincides with $\pi$. 
    
   Let $F_{\bullet}$ be the filtration corresponding to $\pi$.    %the restriction $\mu_{v,\pi}\vert_{\pi_i}$ is simple for each $i=1,\dots,s$ and $s=codim(\F)$. 
   By \cite[Prop.\,2.5]{amini2024tropicalizationlinearseriestilings} and $\mathring{\F}\cap \mathring{\Omega}_{r+1}\neq\emptyset$, we have $\mu_v(F_j)\neq 0$ for all $j=1,\dots,m-1$. Then,
    $$\bigcap_{u\in H-F_j}\ker(\varphi^v_u)\neq 0 \text{ for all }j=1,\dots,m-1.$$ 
    
    It follows from \cite[Lem.\,8.1]{Esteves22072025} that $$\bigcap_{u\in H-F_j}\mathrm{Ess}(\gamma^v_u)\neq \emptyset \text{ for all }j=1,\dots,m-1.$$
    
Now, for each $j=1,\dots, m-1$, define $$I_j:=\bigcap_{u\in H-F_j}\ess(\gamma^v_u).$$  Let $w_{j+1}$ denote the shadow of $I_j\cdot v$ in $H$. Since $\V$ is pure, $\varphi^{w_{j+1}}_{I_j\cdot v}$ is an isomorphism. Then, by \cite[Lem.\,8.1]{Esteves22072025}, we have \begin{equation}
    \bigcap_{u\in H-F_j}\ker(\varphi^v_u)=\ker(\varphi^v_{I_j\cdot v})=\ker(\varphi^v_{w_{j+1}}). \tag{3.2}\label{eq:3.5}
\end{equation} 
Thus, $$ W^{F_j}_v=\Psi^v_H\left(\ker(\varphi^v_{w_{j+1}})\right).$$
Since $F_{j}\subseteq F_{j+1}$ we have  $$I_{j}=\bigcap_{u\in H-F_j} \ess(\gamma^v_u)\subseteq \bigcap_{u\in H-F_{j+1}} \ess(\gamma^v_u)=I_{j+1}.$$
    Note that this inclusion is strict. Indeed, assume by contradiction that $I_j=I_{j+1}$. Then, $w_{j+1}=w_{j+2}$, %\ker(\varphi^v_{w_{j+1}})=\ker(\varphi^v_{w_{j+2}})$
    whence   $$\mu_v(F_j)=\dim\ker(\varphi^v_{w_{j+1}})=\dim\ker(\varphi^v_{w_{j+2}})=\mu_v(F_{j+1}).$$
    Now, for each $q\in \F$, we have $$\mu_v(F_j)=q(F_j)\leq q(F_{j+1})=\mu_v(F_{j+1}),$$ 
    which implies $q(F_{j+1}-F_j)=0$ for all $q\in \F$. This contradicts the assumption  $\mathring{\F}\cap \mathring{\Omega}_{r+1}\neq\emptyset$.

   Set $w_1:=v$. By construction, $w_1,w_2,\dots,w_{m}$ form an oriented $m$-gon, which we denote by $\Delta$. Let $\pi'$ be the partition of $H$ induced by $\Delta$ and $v$. We will show that $\pi'$ coincides with $\pi$. 

    %Let $R_{i}$ denote the intersection of the shadow region $R_{w_i}^\Delta$ of $w_i$ in $\Delta$ and $H$, i.e, $R_i=R_{w_i}^\Delta\cap H$. 

    We claim that for each $i=1,\dots, m$ and $u\in \pi_i$, the shadow of $u$ in $\Delta$ is $w_i$. If the claim holds, then by definition of $\pi'$, we have $\pi_i\subseteq \pi_i'$. It follows that $\pi=\pi'$, finishing the proof of the statement.

%    In order to show the claim, we need only to show that for each $j=1,\dots, m$, there exists an admissible path connecting $w_j$ to $u$ passing through $w_i$. If $j\leq i$, it follows by construction that there is an admissible path from $v$ to $w_i$ passing through $w_j$. Hence, $\ess(\gamma^{w_j}_{w_i})\subseteq \ess(\gamma^v_{w_i})$. Since there is admissible path connecting $v$ to $u$ passing through $w_i$, we obtain the required path by concatenating $\gamma^{w_j}_{w_i}$ with $\gamma^{w_i}_u$.    Suppose that $j\geq i$. Assume by contradiction that there is no such path. It follows from \Cref{lem:3.2} that there is an admissible path connecting $w_i$ to $u$ passing through $w_{i+1}$. Thus, $\ker(\varphi^{w_i}_{w_{i+1}})\subseteq \ker(\varphi^{w_i}_u)$. On the other hand, by construction, there is an admissible path from $v$ to $u$ passing through $w_i$. 
   
   Let $u\in \pi_i$ and $w_j$ its shadow in $\Delta$. Assume by contradiction that $j\neq i$. By \Cref{lem:3.2} there is an admissible path connecting $w_i$ to $u$ passing through $w_{i+1}$. Then $\ker(\varphi^{w_i}_{w_{i+1}})\subseteq \ker(\varphi^{w_i}_u)$. 
Since $w_1,\dots, w_m$ is an oriented $m$-gon, we have
$$\mu_{v,\pi}^*(u)=\dim \varphi^v_u(\ker(\varphi^{v}_{w_{i+1}}))=\dim \varphi^{w_i}_u\varphi^v_{w_i}(\ker(\varphi^{v}_{w_{i+1}}))=\dim \varphi^{w_i}_u(\ker(\varphi^{w_i}_{w_{i+1}}))=0,
$$
where the first equality follows from \eqref{eq:3.5}, the second from $I_{i-1}\subseteq\ess(\gamma^v_u)$, and the third from \Cref{lem:3.3}. 
This contradicts the assumption that $\Pf_{v,\pi}\cap\mathring{\Omega}_{r+1}\neq \emptyset$. 
\end{proof}

%% file: linkedpro.tex
\section{Tilings induced by linked nets.}\label{section:4}

Let $\V$ be a nontrivial finitely generated exact linked net of vector spaces over a $\Z^n$-quiver $Q$, and let $H$ be its minimum set of generators. As in \Cref{section:3}, for each $v,u\in H$, we fix a choice of an admissible path $\gamma^v_u$ connecting $v$ to $u$. For each $v\in H$, we put $(\mu_v,\mu_v^*)$ the corresponding modular pair, as defined in \Cref{subsection:3.1}.
Let $\mathcal{C}_H(\V)$ be the collection of all supermodular functions $\mu_v$ for $v\in H$. Define $$\Pf_\V:=\Pf_{\mathcal{C}_H(\V)}=\bigcup_{v\in H} \Pf_{v}.$$

\begin{lem} \label{lem:4.1}
    Notation as above. For each $v,u\in H, v\neq u$, there exists a nonzero $c\in \k^H$ such that $cW_v\subseteq W_u$. Also, $\mu_v\neq \mu_u$. Finally, the collection $\mathcal{C}_H(\V)$ is separated.
\end{lem}
\begin{proof}
     Let $\gamma$ be an admissible path connecting $v$ to $u$. Let $I$ be the set of $w\in H$ such that the concatenation $\gamma^u_w\gamma$ is admissible. Since $H$ is minimum, for each $w\in I$, we have $\varphi_{\gamma^u_w}\varphi_\gamma\neq 0$ and $\varphi_{\gamma^v_w}\neq 0$. Then, there exists $c_w\in \k^*$ such that $$c_w\varphi_{\gamma^v_w}=\varphi_{\gamma^u_w}\varphi_\gamma.$$ 
Put $c_w:=0$ for $w\in H-I$, and let $c:=(w_w)_{w\in H}\in \k^H$. We claim that $cW_v\subseteq W_u$. Indeed, an element of $W_v$ is of the form $(\varphi_{\gamma^v_w}(s))_{w\in H}$ for $s\in V_v$. Note that $(\varphi_{\gamma^u_w}\varphi_{\gamma}(s))_{w\in H}\in W_u$. If $w\not\in I$, then $\varphi_{\gamma^u_w}\varphi_{\gamma}(s)=0$. If $w\in I$, then $$\varphi_{\gamma^u_w}\varphi_{\gamma}(s)=c_w\varphi_{\gamma^v_w}(s).$$
Hence,  $c(\varphi_{\gamma^v_w}(s))_{w\in H}=(\varphi_{\gamma^u_w}\varphi_{\gamma}(s))_{w\in H}\in W_u$, as required. Note that $c_v=0$ and $c_u\neq 0$. Hence, it follows from \Cref{prop:1.2} that $\mu_v\neq \mu_u$.

Note that $c$ induces an injective map $W_{v,I}\rightarrow W_{u}^I$, whence $\mu_v^*(I)\leq \mu_u(I)$. Then $(I,I^c)$ is a nontrivial separation for $(\mu_u,\mu_v)$.
\end{proof}

\begin{thm}\label{thm:4.2}
The collection of polytopes $\Pf_{\mu_v}$ and their faces, associated to an exact finitely generated nontrivial linked net of vector spaces $\V$ of dimension $r+1$ gives a polyhedral tiling of the standard simplex $\Omega_{r+1}$. In particular, $\Pf_\V=\Omega_{r+1}$. 
\end{thm}

\begin{proof} 

We will check the conditions in \Cref{prop:1.3}. Condition $(2)$ follows from \Cref{lem:4.1}.

As for condition $(1)$, let $v\in H$ and let $\pi=(\pi_1,\pi_2)$ be a bipartition of $H$ such that $\Pf_{v,\pi}$ intersects $\mathring{\Omega}_{r+1}
$. By \Cref{prop:3.6} there is a $2$-gon $\Delta=\lbrace v,u\rbrace\subseteq H$ such that the ordered partition induced by $\Delta$ and $v$ coincides with $\pi$. Now, the bipartition $\pi^c$ is equal to the bipartition induced by $\Delta$ and $u$. It follows from \Cref{lem:3.5} that 
$$\Pf_{v,\pi}=\Pf_\Delta=\Pf_{u,\pi^c}.$$
Hence, by \cite[Prop.\,2.7]{amini2024tropicalizationlinearseriestilings}, $\mu_{v,\pi}=\mu_\Delta=\mu_{u,\pi^c}$. Thus, $\mathcal{W}$ is complete for $\mathring{\Omega}_{r+1}$.

\end{proof}

\subsection{The linked projective space.} 
Let $\V$ be a representation of a $\Z^n$-quiver $Q$ in the category of nontrivial finite-dimensional vector spaces. We denote by $\LP(\V)$ the quiver Grassmanian of subrepresentations of pure dimension $1$ of $\V$. Assume that $\V$ is a finitely generated exact linked net. Let $H$ be the minimum set of generators of $\V$, and let $\LP_H(\V)$ be the subscheme of $\prod_{v\in H}\P(V_v)$ defined by 

$$\LP_H(\V):=\left\lbrace (s_v \vert v\in H)\in \prod_{v\in H} \P(V_v) \,\,\Bigg\vert\,\, \varphi^v_w(s_v)\wedge s_w=0 \text{ for all } v,w\in H\right\rbrace.$$

The natural map $\phi^\V_H\colon\LP(\V)\rightarrow \LP_H(\V)$ induced by restriction is a bijection. See for further details \cite[Prop.\,5.5 and Def.\,5.6]{esteves2024quiverrepresentationsarisingdegenerations}. 

%Let $H_1$ and $H_2$ be collections of vertices of $Q$ $1$-generating $\V$. Assume that $P(H_1)=H_1$. Let$$\Tilde{\Psi}^{H_1}_{H_2}:\prod_{v\in H_1} \P(V_v)\rightarrow \prod_{v\in H_2} \P(V_v)$$sending $(s_v\vert v\in H_1)$ to $(t_v\vert v\in H_1)$ where $t_v=\varphi^{w_v}_v(s_{w_v})$ for each $v\in H_2$ where $w_v$ is the shadow of $v$ in $H_1$. It is a well-defined scheme-theoretical morphisms and restricts to a morphisms $\Psi^{H_1}_{H_2}:\LP_{H_1}(\V)\rightarrow \LP_{H_1}(\V)$ that is an isomorphisms provided $P(H_2)=H_2$. 

\begin{defn}
Let $\V$ be a nontrivial finitely generated exact linked net of vector spaces over a $\Z^n$-quiver $Q$. Give $\LP(\V)$ the scheme structure induced from the bijection $\Psi^\V_H$ where $H$ is the minimum set of generators of $\V$. We call $\LP(\V)$ the \textit{linked projective space} associated to $\V$. 
\end{defn}
It was shown in \cite{esteves2024quiverrepresentationsarisingdegenerations} that $\LP(\V)$ is generically smooth and a local complete intersection. Hence, it is reduced. The irreducible components of $\LP(\V)$, endowed  with the reduced induced scheme structure, can be described as the scheme-theoretic images of the rational maps 
$$\P(V_v)\dashrightarrow \prod_{u\in H} \P(V_u)$$
for $v$ ranging over $H$. It follows from \Cref{lem:2.5} that each map is indeed defined on a dense open subset. Its image is the closure of the open subset $$\LP(\V)_v^*:=\lbrace \mathfrak{M}\subseteq \V 
\,\big\vert\, \mathfrak{M} \text{ is generated by } v\rbrace,$$ which we denote by $\LP(\V)_v$.

\begin{rem}
    In analogy with \cite{Cartwright2011}, we define the \textit{reduction complex} of $\LP(\V)$ as the simplicial complex with one vertex for each component of $\LP(\V)$, where a set of vertices forms a simplex if and only if the intersection of the corresponding components is nonempty. 
    
    By \cite[Prop.\,6.3]{esteves2024quiverrepresentationsarisingdegenerations} the reduction complex of $\LP(\V)$ isomorphic to the simplicial complex formed by the polygons in $H$. In addition, by \Cref{prop:3.6} the latter is isomorphic to the simplicial complex whose set of vertices is $H$, where a set of vertices form a simplex if and only if the intersection of the corresponding polytopes is contained in $\mathring{\Omega}_{r+1}$. 
 \end{rem}
 
\subsubsection{Chow ring.}\label{subsubsection:4.1.1 } Put $\P:=\prod_{v\in H}\P(V_v)$. Let $A^*(\P)$ be the Chow ring of $\P$. Recall that
$$A^*\left(\P\right)=\frac{\Z[h_v\, \vert,\,v\in H]}{(h_v^{r+1} \vert v\in H)}$$
where the $h_v$ denote the pullbacks, via the projection maps, of the hyperplane classes on the $\P(V_v)$. 

Given a closed subscheme $X\subseteq \P$ with pure dimension $r$, its Chow class in $A^*(\P)$ is given by 
$$[X]=\sum_{q\in \Omega_{r}(\Z)} \deg_\P^q(X) \prod_{v\in H} h_v^{r-q(v)}\in A^*\left(\P\right)$$
where $\Omega_{r}(\Z):=\Omega_r\cap \Z^H$, and $\deg^q_\P(X)$ is equal to the number of points in the intersection of $X$ with a product $\prod L_v$  of general linear subspaces $L_v\subset \P(V_v)$ of dimension $r-q(v)$ for $v\in H$. In particular, the Chow class of the small diagonal in $\prod_{v\in H}\P(V_v)$ is  $$\sum_{q\in \Omega_{r}(\Z)} \prod_{v\in H} h_v^{r-q(v)}\in A^*\left(\P\right).$$

\subsubsection{Chow class of the linked projective space.}\label{subsubsection:4.1}
Here we show \Cref{thm:4.4}, which gives a formula for the Chow class of the linked projective space. 
\begin{thm}\label{thm:4.4}
    Let $\V$ be a nontrivial exact finitely generated linked net of vector spaces of dimension $r+1$ over a $\Z^n$-quiver $Q$ and let $H$ be its minimum set of generators. Then $H$ is convex and $$[\LP(\V)]=\sum_{q\in \Omega_{r}(\Z)} \prod_{v\in H} h_v^{r-q(v)}\in A^*\left(\prod_{v\in H}\P(V_v)\right).$$
     
\end{thm}
\begin{proof} After extending the field, we may assume $\k$ is infinite. We will show that, for each $q\in \Omega_r(\Z)$, the monomial $$ \prod_{v\in H}h_v^{r-q(v)}$$
 appears in the class $[\LP(\V)]$ with multiplicity $1$. Since $\LP(\V)$ is reduced, its Chow class can be written as
$$[\LP(\V)]=\sum_{v\in H} [\LP(\V)_v].$$  
 Set $$M_v:=\left\lbrace q\in \Omega_r(\Z) \hspace{0.2cm}\vert\hspace{0.2cm} q(I)\leq\mu_v^*(I)-1 \text{ for each } I\subsetneq H, I\neq\emptyset \right\rbrace.$$
 It follows from \cite{10.1093/imrn/rnx003} that
 $$[\LP(\V)_v]=\sum_{q\in M_v} \prod h_v^{r-q(v)}.$$
 It is thus enough to show that $\bigcup_{v\in H} M_v=\Omega_r(\Z)$ and $M_v\cap M_u=\emptyset$ for each $u,v\in H$ with $u\neq v$.  
 
 Let $q\in \Omega_r(\Z)$. Then $q':=(r+1)q/r\in \Omega_{r+1}$. It follows from \Cref{thm:4.2} that there exists $v\in H$ such that $q'\in \Pf_v$. We claim that $q\in M_v$. We have to show that $q(I)\leq \mu_v^*(I)-1$ for each $I\subsetneq H, I\neq \emptyset$.
%  If $\mu^*_v(I)=r+1$, then $q(I)\leq r= \mu^*(I)-1$. Assume $\mu_v^*(I)<r+1$.
Since $\mu_v$ is simple, we have $\mu_v^*(I)\neq 0$. As $q'\in \Pf_v$, we have $$q(I)\leq r\mu_v^*(I)/(r+1)<\mu^*_v(I).$$ Since $q(I)\in\Z$, we have $q(I)\leq \mu_v^*(I)-1$, as required. 

Assume now by contradiction that $q\in M_u\cap M_v$. By \Cref{lem:4.1}, there exists a bipartition $\pi=(I,J)$ such that $\mu_v^*(I)+\mu_u^*(J)\leq r+1$. Hence, we have
  $$r=q(I)+q(J)\leq (\mu_v^*(I)-1)+(\mu_u^*(J)-1)\leq r-1,$$
 a contradiction. Thus $M_u\cap M_v=\emptyset$.

\end{proof}

%\section{Complementary results and discussions.}